\documentclass[12pt]{article}
\usepackage{amssymb}
\usepackage{amsfonts}
\usepackage{amsmath}
\usepackage[usenames]{color}
\usepackage{mathrsfs}
\usepackage{amsfonts}
\usepackage{amssymb,amsmath}
\usepackage{CJK}
\usepackage{cite}
\usepackage{cases}
\usepackage{amsthm}

\def\cl{\centerline}

\def\vs{\vspace*}

\def\ni{\noindent}

\numberwithin{equation}{section}
\newtheorem{theo}{Theorem}[section]
\newtheorem{defi}[theo]{Definition}
\newtheorem{coro}[theo]{Corollary}
\newtheorem{lemm}[theo]{Lemma}

\newtheorem{rema}[theo]{Remark}

\newtheorem{remark}[theo]{Remark}

\begin{document}
\begin{center}
\cl{\large\bf \vs{6pt} Isoparametric hypersurfaces in Finsler space forms\,{$^*\,$}}
\footnote {$^*\,$ Project supported by NNSFC (No.11471246, 11971253), AHNSF (No.1608085MA03).
\\\indent\ \ $^\dag\,$ yst419@163.com
}
\cl{ Qun He$^{1}$, Yali Chen$^1$, Songting Yin$^{2,\dag}$, Tingting Ren$^{1}$}

\cl{\small 1. School of Mathematical Sciences, Tongji University, Shanghai
200092, China}
\cl{\small 2. Department of Mathematics and computer science, Tongling University,}
\cl{ Tongling, 244000, China}
\end{center}

{\small
\parskip .005 truein
\baselineskip 3pt \lineskip 3pt

\noindent{{\bf Abstract:}
In this paper, we study isoparametric hypersurfaces in Finsler space forms by investigating focal points, tubes and parallel hypersurfaces of  submanifolds. We prove that the focal submanifolds of isoparametric hypersurfaces are anisotropic-minimal and obtain the Cartan-type formula in a Finsler space form with vanishing reversible torsion, from which we give some classifications on the number of distinct principal curvatures or their multiplicities.
\vs{5pt}

\ni{\bf Key words:}
 Finsler space form, isoparametric hypersurfaces, focal submanifolds, Randers space, principal curvature, anisotropic mean curvature.}

\ni{\it Mathematics Subject Classification (2010):} 53C60, 53C42, 34D23.}
\parskip .001 truein\baselineskip 6pt \lineskip 6pt
\section{Introduction}
In Riemannian geometry, the study on isoparametric hypersurfaces has a long history. Since 1938, E.Cartan began to study the isoparametric hypersurfaces in real space forms with constant sectional curvature $c$ systematically. One of excellent works Cartan did is that he proved if~$k_{1}, k_{2}, \cdots, k_{g}$ are the all distinct principal curvatures, then they satisfy the following formula
\begin{equation}\label{0.0}
\sum_{i\neq j} m_{j}\frac{c+k_{j}k_{i}}{k_{j}-k_{i}}=0,~~~~~~j=1, \cdots, g,
\end{equation}
where~$m_{i}$ is the multiplicity of~$k_{i}$. (\ref{0.0}) is known as Cartan formula, which plays a vital role in the classification of
isoparametric hypersurfaces in space forms. Those in Euclidean and hyperbolic spaces were classified in 1930's~\cite{C,SF,S}. For the most difficult case, those in a unit sphere, were recently completely solved~\cite{C1}. It is a  natural idea to generalize the theories to Finsler geometry.

In Finsler geometry, the conception of isoparametric hypersurfaces has been introduced in~\cite{HYS}. Let~$(N, F)$ be an
$n$-dimensional Finsler manifold. A function~$f$ on~$(N, F)$ is said to be \textit{isoparametric} if there are~two functions $a (t)$ and~$b (t)$ such that
\begin{equation}\label{0.01} \left\{\begin{aligned}
&F(\nabla f)={a}(f),\\
&\Delta f={b}(f),
\end{aligned}\right.
\end{equation}
where~$\nabla f$ denotes the gradient of~$f$, which is defined by means of the Legendre transformation, and~$\Delta f$ is a nonlinear
Finsler-Laplacian of $f$. Each regular level hypersurface of an isoparametric function is called an \emph{isoparametric hypersurface} (see Section 2.1 and Section 2.3 for details). For simplicity, we mainly consider $\hat{\Delta} f$ in this paper, which is the Laplacian of $\hat{g}=g_{\nabla f}$ and
independent from the choice of volume elements.

The flag curvature of a Finsler manifold is a natural generalization of the sectional curvature in Riemannian geometry. Similarly, we can call a complete and simply connected Finsler manifold with constant flag curvature a \emph{Finsler space form}. If a Finsler manifold is only \emph{forward} (resp. \emph{backward}) complete, we call it \emph{forward} (resp. \emph{backward}) complete \emph{Finsler space form}. In this paper, we denote an $n$-dimensional Finsler space form with constant flag curvature~$c$ by ${N}(c)$ and denote a forward (resp. backward) complete Finsler space form by $\overrightarrow{N}(c)$ (resp. $\overleftarrow{N}(c)$) for the sake of simplicity. Studying and classifying isoparametric hypersurfaces in Finsler space forms are interesting problems naturally generalized from Riemannian geometry. Unlike the Riemannian geometry, there are infinitely many Finsler space forms, which are not isometric or even not homothetic to each other, and they are far from being completely classified.  People know very little about Finsler space forms except for some special cases.  Even for those known Finsler space forms with non-zero flag curvature, they may be very complicated, like the examples constructed on spheres~\cite{B} by R. Bryant. Therefore, the classification of isoparametric hypersurfaces in Finsler space forms is an arduous and complex task.

In~\cite{HYS,GM,HYS1,X,DH}, the isoparametric hypersurfaces in some special Finsler space forms including in  Minkowski spaces (with zero flag curvature), Funk type spaces (with negative constant flag curvature) had been completely classified. For ambient spaces with positive constant flag curvatures, Xu \cite{X} studied a special class of isoparametric hypersurfaces in the (non-Riemannian) Randers sphere.
In this paper, we consider isoparametric hypersurfaces in more general Finsler space forms.

From \cite{HYS}, we know that in a Finsler space with constant flag curvature,
 all principal curvatures of an isoparametric hypersurface are also constant. Therefore, the first step to classify isoparametric hypersurfaces is to derive Cartan type formula. Unfortunately, it is very difficult to derive Cartan-type formula by Cartan's way in general Finsler space forms.  As an alternative to Cartan's  measure, one can also derive (\ref{0.0}) by studying tubes, parallel hypersurfaces and focal submanifolds of isoparametric hypersurfaces in $\mathbb{S}^{n}$ (see \cite{CR}), which is seen as an embedded hypersurface of $\mathbb{R}^{n+1}$. But there are still some barriers in this way because in general, a Finsler space form may not be isometrically embedded in a Minkowski space. So many of research methods in Riemannian geometry are no longer applicable.

  In this paper, we will try to use new ideas to derive Cartan-type formula.
Firstly, in Section 3, we introduce focal points, tubes and parallel hypersurfaces of anisotropic submanifolds in a Finsler space form and study their principal curvatures by using the theories of geodesics and Jacobi fields. In Section 4,
by using the theories of focal manifolds, we derive Cartan-type formula and get some classifications of isoparametric hypersurfaces in~$\overrightarrow{N}(c)$ with vanishing \emph{reversible torsion} $\tau$ or $\tau'$, which are defined by (\ref{3.19}) and (\ref{3.191}) respectively.

\begin{theo}\label{thm00}$\mathbf{(General~Cartan~formula)}$\label{thm4} Let $M\hookrightarrow N (c)$ be a connected isoparametric hypersurface with $g$ distinct principal curvatures
$\lambda_{1},~\lambda_{2},~\cdots, ~\lambda_{g},$
with respective multiplicities of $m_{i}$. If $\tau=0$ or $\tau'=0$, then for each $i$, $1\leq i\leq g$, there exists a set of positive constants $k_{j}\geq m_j$ such that the following formula holds
\begin{equation}\label{c1}
\sum_{j\neq i}k_{j}\frac{c+\lambda_{i}\lambda_{j}}{\lambda_{i}-\lambda_{j}}=0.
\end{equation}
Especially, $k_{j}=m_j$ when $\tau=0$. In this case, each focal submanifold~of~$M$ is anisotropic-minimal.
\end{theo}

\begin{theo} \label{thm0}
Let~$M$ be a connected isoparametric hypersurface in a Finsler space form~$\overrightarrow{N}(c)$ with  vanishing reversible torsion $\tau$ or $\tau'$.  \\
$(1)$~If $c\leq0,$ then $M$ has two distinct principal curvatures at most; \\
$(2)$~If $c>0$ and $\tau=0$, then the distinct principal curvatures $\lambda_{i}=\sqrt{c} \cot \theta_{i}$ satisfy
\begin{align}
\theta_{i}=\theta_{1}+ (i-1)\frac{\pi}{g},~~~~ 1\leq i \leq g, \label{3.201}
\end{align}
where $0<\theta_{1}<\cdots< \theta_{g}< \pi$ and their multiplicities satisfy $m_{i}=m_{i+2}$ (subscripts mod $g$). Thus, all of the principal curvatures have the same multiplicity if $g$ is odd and there are at most two distinct multiplicities if $g$ is even. Besides, for any $x\in M$, there are $2g$ focal points of $(M,x)$ at least along every normal geodesic to $M$  through $x$, and they are evenly distributed at intervals of length $\frac{\pi}{\sqrt{c}g}$.
\end{theo}

\begin{remark} From (\ref{3.19}) and (\ref{3.191}), it is obvious that $\tau=\tau'=0$ always holds if~$F$ is reversible~and~$\tau'=0$ always holds if $F$ is Berwald. Therefore, Theorem \ref{thm00} and Theorem \ref{thm0} not only generalize the corresponding results in Riemannian geometry, but also generalize the corresponding results in \cite{HYS,GM}.
\end{remark}
Finally, in Section 5, we will prove that if $F$ is a Randers metric and its navigation vector field is Killing, then~$\tau=0$ always holds. Randers metrics are nonreversible and play a fundamental role in Finsler geometry. There exist a large number of nontrivial Randers space forms. Until now, the most known examples of Finsler space forms with non-zero flag curvature are Randers space forms. From above, we have
\begin{coro} \label{thm02}
In a Minkowski space or a Randers space form whose navigation vector field is Killing, the conclusions of Theorem \ref{thm00} and Theorem \ref{thm0} always hold not only for isoparametric hypersurfaces, but also for $d\mu_{BH}$-isoparametric hypersurfaces.
\end{coro}

\section{Preliminaries}
\subsection{Finsler Laplacians}
Let $(N,F)$ be an $n$-dimensional oriented smooth Finsler manifold and $TN$ be the tangent bundle over $N$ with local coordinates
$(x,y)$, where $x=(x^1,\cdots ,x^n)$ and $y=(y^1,\cdots ,y^n)$.
Here and from now on, we will use the following convention of index ranges unless other stated:
$$1\leq i, j,\cdots \leq n ;~~~~~~~1\leq a, b, \cdots \leq m< n;$$
$$m+1\leq \alpha, \beta,\cdots  \leq n.$$
The fundamental form $g$ of $(N,F)$ is given by
$$g:=g_{ij}(x,y)dx^{i} \otimes dx^{j}, ~~~~~~~g_{ij}(x,y)=\frac{1}{2}[F^{2}] _{y^{i}y^{j}}.  $$
The curve $\gamma(t)$ is called a geodesic in~$N$, if its local coordinates $(x^i(t))$ satisfy
\begin{align*}
\ddot{x}^{i}(t)+2G^{i}\left(x(t),\dot{x}(t)\right)=0,
\end{align*}
where
\begin{align*}
G^{i}=\frac{1}{4}g^{il}\left\{[F^{2}]_{x^{k}y^{l}}y^{k}-[F^{2}]_{x^{l}}\right\}
\end{align*}
 are the geodesic coefficients of $(N,F)$. Using the geodesic coefficients, we can define a torsion-free connection $\nabla$ on the pull-back bundle $\pi^{\ast}TN$ by (\cite {BCS}, p.39)
\begin{align*}\nabla
\frac{\partial}{\partial x^i}=\omega_{i}^{k}\frac{\partial}{\partial
x^{k}}=\Gamma^{k}_{ij}dx^j\otimes\frac{\partial}{\partial x^{k}},~~
{\Gamma}^{k}_{ij}=\frac{\partial G^{k}}{\partial y^{i}\partial y^{j}}.
\end{align*} It is called \emph{Berwald connection} and satisfies
\begin{align}
  dg_{ij}&-g_{ik} {\omega^{k}_{j}}-g_{kj} {\omega^{k}_{i}}=2FC_{ijk}\delta y^{k}-2L_{ijk}dx^{k},\label{eq.con2}
\end{align}
where $\delta y^{i}:=\frac{1}{F}(dy^{i}+N^{i}_{j}dx^{j})$, $N_j^i:=\frac{\partial G^i}{\partial y^j}=\Gamma^{i}_{jk}y^k$, $C_{ijk}:=\frac{1}{2}\frac{\partial g_{ij}}{\partial y^k}$ is the \emph{Cartan tensor} and $L_{ijk}:=F\dot {C}_{ijk}={C}_{ijk|l}y^l$ is the
\emph{Landsberg curvature}. The curvature 2-forms of the
Berwald connection are
\begin{eqnarray*}
d{ \omega}^{i}_{j}-{ \omega}^{k}_{j}\wedge { \omega}^{i}_{k}:=\frac{1}{2}{ R}^{i}_{j~kl}dx^{k}\wedge
dx^{l}+{ P}^{i}_{j~kl}dx^{k}\wedge\delta y^{l}.
\end{eqnarray*}
The {\it flag curvature tensor} can be written as
$$R^i_{~k}:=\ell^j{ R}^{i}_{j~kl}\ell^l,\qquad R_{jk}=g_{ij}R^i_{~k},$$
where $\ell=\frac{y}{F}$ is the vector field dual to the Hilbert form $\omega=[F]_{y^i}dx^i$.  For a vector
$V=V^{i}\frac{\partial}{\partial x^i}$ satisfying $g_{ij}V^iV^j=1$ and $g_{ij}y^iV^j=0$, the \emph{flag curvature} of $(N,F)$ is defined by
\begin{align*}K(y;V)=R_{ij}V^{i}V^{j}. \end{align*}

Let $X=X^{i}\frac{\partial}{\partial x^{i}}$ be a differentiable vector field. Then the \emph{covariant derivatives} of $X$ along
$v=v^i\frac{\partial}{\partial x^{i}}\in T_{x}N$ with respect to a reference  vector $w\in T_{x}N\backslash 0$ for  the Berwald connection is defined by
\begin{align}
{ D}^{w}_{v}X(x):&=\left\{v^{j}\frac{\partial X^{i}}{\partial x^{j}}(x)+{ \Gamma}^{i}_{jk}(w)v^{j}X^{k}(x)\right\}\frac
{\partial}{\partial x^{i}}.\label{Z1}\end{align}

Let ${\mathcal L}:TN \to T^{\ast}N$ denote the \emph{Legendre transformation}, which satisfies ${\mathcal L}(\lambda
y)=\lambda {\mathcal L}(y)$ for all $\lambda>0$ and~$y\in TN$. Moreover, we know from {\cite{SZ}} (p.38-39) that
\begin{align}\mathcal L(y)&=F(y)[F]_{y^{i}}(y)dx^i,~~\forall y\in TN\setminus \{0\},~~\mathcal L(0)=0,\nonumber\\
\mathcal L^{-1}(\xi)&=F^*(\xi)[F^*]_{\xi_{i}}(\xi)\frac{\partial}{\partial x^i},~~\forall \xi\in T^*N\setminus \{0\},~~\mathcal L^{-1}(0)=0,\label{Z01}\end{align}
 where $F^*$ is the dual metric of $F$.
For a smooth function $f:N \to \mathbb{R}$, the \emph{gradient vector} of $f$ at $x$ is defined as $\nabla f(x):={\mathcal
L}^{-1}(df(x))\in T_{x}N$.
Set $N_{f}:=\{x\in N|df(x)\neq 0\}$ and $\nabla^{2}f(x):=D^{\nabla f}(\nabla f)(x)$ for $x\in N_{f}$. We define the \emph{Laplacian} of~$f$
by
\begin{equation}\label{l1}
\hat{\Delta} f=\textmd{tr}_{g_{_{\nabla f}}}(\nabla^{2}f).
\end{equation}
And the \emph{Laplacian} of~$f$ with respect to the volume form~$d\mu=\sigma(x)dx=\sigma(x)dx^{1}\wedge dx^{2}\wedge\cdots\wedge dx^{n}$ can be represented as \begin{align}\Delta_{\sigma} f=\textmd{div}_{\sigma}(\nabla f)=\frac{1}{\sigma}\frac{\partial}{\partial x^{i}}(\sigma g^{ij}(\nabla f)f_{j})=\hat{\Delta} f-S(\nabla f),
\label{l2}\end{align}
where
\begin{align}\label{1.4.19}
S(x,y)=\frac{\partial G^{i}}{\partial y^{i}}-y^{i}\frac{\partial}{\partial x^{i}}(\ln \sigma(x))
\end{align}
is the \emph{$\mathbf{S}$-curvature}.
\subsection{Reversible torsions}
In general, $\mathcal L^{-1}(-\xi)\neq-\mathcal L^{-1}(\xi)$.
So for any $\eta\in S_xN=\{X\in T_xN~|~F(X)=1\}$, we denote $$\eta_{-}=\frac{{\mathcal L}^{-1}(-{\mathcal L}\eta)}{F^{*}(-{\mathcal L}\eta)}.$$
If $F$ is reversible, then $\eta_{-}=-\eta$.
\begin{defi} \label{def0} Let ${\eta}$ be a unit vector field in the neighbourhood of $x\in N$ and $X\in T_{x}N$ satisfy $g_\eta(\eta,X)=0$.
Set
\begin{align}
\tau_{\eta}(X):&=D^{\eta}_{X}\eta+D^{\eta_{-}}_{X}\eta_{-},\label{3.19}\\
\tau'_{\eta}(X):&=D^{\eta}_{X}\eta-D^{\eta_{-}}_{X}\eta.\label{3.191}
\end{align}

We call $\tau_{\eta}$ and $\tau'_{\eta}$ the \textbf{first class} and
\textbf{second class reversible torsion} of $\eta$ in $(N,F)$ respectively.
\end{defi}

\subsection{Isoparametric hypersurfaces}

Let $f$ be a non-constant $C^1$ function defined on a Finsler manifold $(N,F,d\mu)$ and smooth on $N_f$. Set $J=f(N_f)$. The function $f$ is
said to be \emph{$d\mu$-isoparametric} (resp. \emph{isoparametric}), if there exist a smooth function ${a} (t)$ and a continuous function $b (t)$ on $J$ such that \eqref{0.01} holds for $\Delta {f}=\Delta_{\sigma} f$ (resp. $\Delta {f}=\hat{\Delta} f$), which is defined by \eqref{l2} (resp. \eqref{l1}).
All the regular level surfaces ${M}_t = {f}^{-1}(t)$ are named an \emph{($d\mu$-)isoparametric family}, each of which is called an \emph{($d\mu$-)isoparametric hypersurface} on $(N,F,d\mu)$.
\begin{rema}\label{111} From \cite{HYS}, we know that if $(N,F,d\mu)$ has constant $\mathbf{S}$-curvature, then~$f$~is isoparametric if and only if it is $d\mu$-isoparametric. Therefore, as long as the condition of constant $\mathbf{S}$-curvature is added, all conclusions in this paper are still valid for $d\mu$-)isoparametric hypersurface.
\end{rema}
\section{Anisotropic submanifolds of a Finsler manifold}
\subsection{Anisotropic mean curvature}

Let $(N,F)$ be an $n$-dimensional Finsler manifold and $\phi: M\to(N, F)$ be an $m$-dimensional immersion. For simplicity, we will denote $d\phi X$ by $X$. Let
$$\mathcal{V}(M)=\{(x,\xi)~|~x\in M,\xi\in T_x^{*}N,\xi (X)=0,\forall X\in T_xM\},$$
which is called the
{\it normal bundle} of $\phi$ or $M$. Let
$$ \mathcal{N}M={\mathcal L}^{-1}(\mathcal{V}(M))=\{(x, \textbf{n})|~x\in \phi(M), \textbf{n}={\mathcal L}^{-1}(\xi), \xi \in \mathcal{V}_{x}(M)\}.$$
Then $\mathcal{N}M\subset TN$. For any $\textbf{n}\in \mathcal{N}M, \lambda>0$ and $ X\in \Gamma(TM)$, we have $\lambda\textbf{n}\in \mathcal{N}M$ and $g_{\textbf{n}}(\textbf{n},X)=0$. So we call $\textbf{n}\in \mathcal{N}M$ \emph{the normal vector} of $M$. We also call $\mathcal{N}M$ the
{\it normal bundle} of $\phi$ or $M$. But in general, it is not a vector bundle. We call $\{(M,g_{\textbf{n}})|\textbf{n}\in\mathcal{N}M\}$ an \emph{anisotropic submanifold} of $(N,F)$ to distinguish it from an isometric immersion submanifold $(M,\phi^*F)$.

Moreover, we denote the {\it unit normal bundle} by
$$\mathcal{V}^0(M)=\{\nu\in \mathcal{V}(M)|F^*(\nu)=1 \},$$
$$ \mathcal{N}^0M=\{\textbf{n}\in \mathcal{N}M|~F(\textbf{n})=1\}={\mathcal L}^{-1}(\mathcal{V}^0(M)).$$
For any $\nu\in \mathcal{V}^0(M)$, set $\nu_{+}=\nu$ and $\nu_{-}=\frac{-\nu}{F^*(-\nu)}$. Then $\nu_{-}\in \mathcal{V}^0(M)$ and $\textbf{n}_{\pm}={\mathcal L}^{-1}\nu_{\pm}\in \mathcal{N}^0(M)$.
 For any  $X\in T_xM$ and $\textbf{n}\in \mathcal{N}^0(M)$,  define the \emph{shape operator} ${A}_{\textbf{n}}:T_xM\rightarrow T_xM$ by the following \emph{Weingarten formula}
\begin{equation}\label{0.1}{A}_{\textbf{n}}(X)=-\left(D^{\textbf{n}}_{X}\textbf{n}\right)^{\top}_{g_{\textbf{n}}},
\end{equation}
where $D$ is the Berwald covariant derivative defined by (\ref{Z1}). Noticed that $A_\textbf{n}(X)$ does not depend on the extension of $\textbf{n}$. In fact,
for any two extensions  $\textbf{n}_1$ and $\textbf{n}_2$ of \textbf{n} and any two vectors $X,Y\in T_xM$,
\begin{equation*}
g_{\textbf{n}}({A}_{\textbf{n}_1}(X),Y)=-g_{\textbf{n}_1}({D}_X^{\textbf{n}_1}\textbf{n}_1,\tilde{Y})=g_{\textbf{n}_1}({\textbf{n}_1},{D}_X^{\textbf{n}_1}\tilde{Y})
=g_{\textbf{n}_2}({\textbf{n}_2},{D}_X^{\textbf{n}_2}\tilde{Y})=g_{\textbf{n}}({A}_{\textbf{n}_2}(X),Y),
\end{equation*}
where $\tilde{Y}$ is a extension of $Y$. Moreover, it is easy to show that
\begin{equation*}g_{\textbf{n}}({A}_{\textbf{n}}(X),Y)
=g_{\textbf{n}}(X,{A}_{\textbf{n}}(Y)),~~~~\forall X,Y\in T_xM.
\end{equation*}
The eigenvalues of ${A}_{\textbf{n}}$,
$\lambda_1,\lambda_2,\cdots,\lambda_{m}$, and $\hat{H}_{\textbf{n}}=\sum_{i=1}^{m}\lambda_{i}$ are called the \emph{ principal
curvatures} and the \emph{anisotropic mean curvature} with respect to $\textbf{n}$, respectively. If $\lambda_1=\lambda_2=\cdots=\lambda_{m}$ for any $\textbf{n}\in \mathcal{N}M$, we call $M$ \emph{anisotropic-totally umbilic}. If $\hat{H}_{\textbf{n}}=0$ for any $\textbf{n}\in \mathcal{N}M$, we call $M$ an \emph{anisotropic-minimal submanifold} of $(N,F)$.

 \subsection{Anisotropic hypersurfaces in Finsler space form}

 Let $(N,F)$ be an $n$-dimensional Finsler manifold and $\phi:M\to N$ be an embedded hypersurface of $(N,F)$. For any $x\in M$, there exist exactly two unit\emph{ normal vectors} $\textbf{n}$ and $\textbf{n}_{-}$.
Let $\textbf{n}$ be a given unit normal vector of~$M$ and set $\hat g=\phi^*g_{\textbf{n}}$. We call $(M,\hat g)$ an oriented \emph{anisotropic hypersurface}. In this paper, all the submanifolds and hypersurfaces are anisotropic, which will be no longer declared for simplicity's sake.

 From (\ref{0.1}) and \cite{HYS1}, we have the following \emph{Gauss-Weingarten formulas} with respect to $g_{\textbf{n}}$ for the Berwald connection
\begin{align}\label{1.1}{ D}^{\textbf{n}}_{X}Y&={ {\hat{\nabla}}}_{X}Y+\hat{h}(X,Y)\textbf{n},\\
D^{\textbf{n}}_{X}\textbf{n}&=-{A}_{\textbf{n}}X,~~~~~\quad X,~Y\in \Gamma(TM) .\label{1.2}\end{align}
Here
\begin{equation}\label{1.3}\hat{h}(X,Y):=g_{\textbf{n}}(\textbf{n},{ D}^{\textbf{n}}_{X}Y)=\hat {g}({A}_{\textbf{n}}X,Y)),\end{equation}
which is called the \emph{second fundamental form}, and
 ${ \hat{\nabla}}$ is a torsion-free linear connection on $M$ satisfying (\cite{HYS1})
\begin{eqnarray}\label{1.4}
({ \hat{\nabla}}_X\hat g)(Y,Z)=-2C_{\textbf{n}}({A}_{\textbf{n}}X,Y,Z)-2L_{\textbf{n}}(X,Y,Z),~~~~~~~X,Y,Z\in \Gamma(TM)
\end{eqnarray}
where $C_{\textbf{n}}$ and $L_{\textbf{n}}$ are the Cartan tensor and the
Landsberg curvature, respectively, with $y={\textbf{n}}$.

\begin{lemm}  For the induced connection ${ \hat{\nabla}}$ on hypersurfaces of a Finsler manifold with constant flag curvature, we have
\begin{align}
({ \hat{\nabla}}_X{A}_{\textbf{n}})Y=({ \hat{\nabla}}_Y{A}_{\textbf{n}})X,~~~~~~X,Y\in \Gamma(TM),\label{1.15}
\end{align}
where $
({ \hat{\nabla}}_X{A}_{\textbf{n}})Y={ \hat{\nabla}}_X({A}_{\textbf{n}}Y)-{A}_{\textbf{n}}({ \hat{\nabla}}_XY).
$
\end{lemm}
\proof
Recall that the Codazzi equation of hypersurfaces of a general Finsler manifold is (\cite{HYS1})
\begin{align}
 g_{\bf{n}}({ R}_{\bf{n}}( X , Y ) Z,\bf{n})&=({ \hat{\nabla}}_X\hat{h})(Y,Z)-({ \hat{\nabla}}_Y\hat{h})(X,Z)\nonumber\\
&+2L_{\bf{n}}(X,{A}_{\bf{n}}(Y),Z)-2L_{\bf{n}}(Y ,{A}_{\bf{n}}(X),Z),\label{1.17}
\end{align}
where\begin{align}({ \hat{\nabla}}_X\hat h)(Y,Z):=X\hat h(Y,Z)-\hat h({ \hat{\nabla}}_XY,Z)-\hat h(Y,{ \hat{\nabla}}_XZ).
\label{1.18}\end{align}
 From \cite{BCS} (p.79), if $N$ has constant flag curvature $c$, then
\begin{eqnarray}\label{1.8}
{ R}^{i}_{j~kl}=c(g_{jl}\delta^{i}_{k}-g_{jk}\delta^{i}_{l}),
\end{eqnarray}
and thus
\begin{eqnarray}\label{1.12}g_{\bf{n}}({ R}_{\bf{n}}( X , Y )Z, \textbf{n})=0,\end{eqnarray}
for any $X, Y, Z \in T_xM.$
Combining (\ref{1.18}), (\ref{1.3}) and (\ref{1.4}) yields
\begin{align}
({ \hat{\nabla}}_X\hat h)(Y,Z)=&X\hat g({A}_{\bf{n}} Y,Z)-\hat g({A}_{\bf{n}} Z, { \hat{\nabla}}_XY)-\hat g({A}_{\bf{n}} Y,{ \hat{\nabla}}_XZ)\nonumber\\
=&\hat g({ \hat{\nabla}}_X({A}_{\bf{n}} Y),Z)-2C_{\textbf{n}}({A}_{\textbf{n}}X, {A}_{\textbf{n}}Y,Z)-2L_{\textbf{n}}(X,{A}_{\textbf{n}}Y,Z)-\hat g({A}_{\bf{n}} Z, { \hat{\nabla}}_XY)\nonumber\\
=&\hat g(({ \hat{\nabla}}_X{A}_{\bf{n}}) Y,Z)+\hat g({A}_{\bf{n}}({ \hat{\nabla}}_X Y),Z)-2C_{\textbf{n}}({A}_{\textbf{n}}X, {A}_{\textbf{n}}Y,Z)\nonumber\\
&-2L_{\textbf{n}}(X,{A}_{\textbf{n}}Y,Z)-\hat g({A}_{\bf{n}} Z, { \hat{\nabla}}_XY)\nonumber\\
=&\hat g(({ \hat{\nabla}}_X{A}_{\bf{n}}) Y,Z)-2C_{\textbf{n}}({A}_{\textbf{n}}X, {A}_{\textbf{n}}Y,Z)-2L_{\textbf{n}}(X,{A}_{\textbf{n}}Y,Z).
\label{1.19} \end{align}
Similarly, we have
\begin{eqnarray}\label{1.20}
({ \hat{\nabla}}_Y\hat h)(X,Z)=\hat g(({ \hat{\nabla}}_Y{A}_{\bf{n}}) X,Z)-2C_{\textbf{n}}({A}_{\textbf{n}}Y, {A}_{\textbf{n}}X,Z)-2L_{\textbf{n}}(Y,{A}_{\textbf{n}}X,Z).
\end{eqnarray}
Substituting (\ref{1.19}) and (\ref{1.20}) into (\ref{1.17}), and using (\ref{1.12}), we obtain (\ref{1.15}).
\endproof

\subsection{Focal points}

 Let $\phi:M \rightarrow \overrightarrow{N}(c)$ be an embedded submanifold. Let $exp: TN\rightarrow N$ be the exponential map of $N$. The normal exponential map
$$E: \mathcal{N}M\rightarrow N$$ is the restriction of the exponential map of $TN$ to $\mathcal{N}M$, that is, $E(x, \textbf{n})=\exp_{x}\textbf{n}$. If $\textbf{n}$ is the zero vector in the tangent space of $N$ at ${x}$, then $E(x, \textbf{n})$ is the point ${x}$. From [13, \S11.1], we know that the exponential map is $C^{\infty}$ on $TN\backslash {0}$ and $C^{1}$ on the zero sections of $TN$, so is $E$.

\begin{defi} \label{def1}
The focal points of $M$ are the critical values of the normal exponential map $E$. Specifically, a point $p\in N$ is called a \textbf{focal point} of $(M, x)$ of multiplicity $m$ if $p=E(x, \textbf{n})$ and the differential $E_{*}$ at $(x, \textbf{n})$ has nullity $m>0$.
\end{defi}

The \emph{focal set} of $M$ is the set of all focal points of $(M, x)$ for all $x\in M$.
Since $\mathcal{V}(M)$ and $N$ have the same dimension and ${\mathcal L}^{-1}:\mathcal{V}(M)\backslash\{0\}\rightarrow\mathcal{N}M\backslash\{0\}$ is a smooth diffeomorphism, it follows from Sard's Theorem that the focal set of $M$ has measure zero in $N$.

We now assume that $\textbf{n}$ is a unit normal vector of $\phi(M)$ at  $x$, and
$E(x,s\textbf{n})=\exp_{x}s\textbf{n}$, where~$s\geq 0$. Then $E(x, s\textbf{n})$ is the point of $N$ reached by traversing a length $s$ along the geodesic in $N$ with initial point ${x}$ and initial tangent vector $\textbf{n}$.
For a fixed $x_{0}\in \phi(M)$, let $U$ be a coordinate neighborhood of $x_0$ in $\phi(M)$ and
 $x=x(t)$ be a curve in $U$ such that
\begin{align*} x(0)=x_{0}, ~~~~\dot{x}(0)=X\in T_{x_{0}}\phi(M). \end{align*}  We consider the smooth variation of the geodesic $\gamma=\gamma(s)=E(x_0, s\textbf{n}(x_0))$:
$$\Phi:[0,+\infty)\times (-\varepsilon, \varepsilon)\rightarrow N$$
\begin{align*} (s,t)\mapsto E(x(t), s\textbf{n}(x(t)))=exp_{x(t)}s\textbf{n}(x(t)) \end{align*}
such that
$$\Phi(s,0)=\gamma(s),s\geq 0,$$
$$\Phi(0,t)=x(t), |t|<\varepsilon.$$

Denote
\begin{align}\widetilde{J}(s,t)=\Phi_{*}(\frac{\partial}{\partial t})=\frac{\partial \Phi}{\partial t}, ~~~\widetilde{T}(s,t)=\Phi_{*}(\frac{\partial}{\partial s})=\frac{\partial \Phi}{\partial s}, \label{2.4}\end{align}
and $$J(s)=\widetilde{J}(s,0),~~~~ T(s)=\widetilde{T}(s,0).$$
We have
\begin{align*}\widetilde{J}(0,t)=\dot{x}(t),~~~~ \widetilde{T}(0,t)=\textbf{n}(x(t)),\end{align*}
and
\begin{align}J(0)=X,~~ T(0)=\textbf{n}(x_{0}),~~g_{T(0)}(J(0), ~~T(0))=0. \label{2.4.1}\end{align}
 For any fixed $t$, $\Phi(s,t)$ is a geodesic in $N$, then $J(s)$ is a Jacobi field along  the geodesic $\gamma(s)$. Note that $N$ is of constant flag curvature. By the property of Jacobi field (\cite{BCS}),
 it follows   from a direct computation that
\begin{align}
\left\{
\begin{aligned}
J(s) & = E_{1}(s)\mathfrak{s}_{c}^{'}(s)+E_{2}(s)\mathfrak{s}_{c}(s), \\
E_{1}(0) & = J(0), \\
E_{2}(0) & = D_{T}^{T}J|_{s=0},
\end{aligned}
\right.
\label{2.9}\end{align}
where \begin{align}
\mathfrak{s}_{c}(s)=\left\{
\begin{array}{rcl}s,       &      & {c=0,}\\
\frac{\sin\sqrt{c}s}{\sqrt{c}},    &      & {c>0,}\\
\frac{\sinh\sqrt{-c}s}{\sqrt{-c}},       &      & {c<0,}
\end{array} \right.
\label{2.2}\end{align} and $E_{i}(s)$ is the parallel vector field along $\gamma(s)$ satisfying $g_T(E_{i}, T)=0,$ for $i=1,2.$

The following lemma gives the location of the focal points of $(M,x)$ along the geodesic $\gamma(s)=E(x,s\textbf{n})$ for $s\in \mathbb{R} $, in terms of the eigenvalues of the shape operator $A_{\textbf{n}}$ at $x$.
\begin{lemm} \label{thm1}
Let $\phi:M  \rightarrow \overrightarrow{N}(c) $ be a immersion submanifold , and let $\textbf{n}$ be a unit normal vector to $\phi(M )$ at ${x}$. Then $p=E(x, s\textbf{n})$ is a focal point
 of $(M ,x)$ of multiplicity $m_0>0$ if and only if there is an eigenvalue $\lambda$ of the shape operator $A_{\textbf{n}}$ of multiplicity $m_0$ such that
\begin{align} \lambda=\frac{\mathfrak{s}_{c}'(s)}{\mathfrak{s}_{c}(s)}=\left\{
\begin{array}{rcl}
\frac{1}{s},     &      & {c=0,}\\
\sqrt{c}\cot \sqrt{c}s,   &      & {c>0,}\\
\sqrt{-c}\coth \sqrt{-c}s,       &      & {c<0.}
\end{array} \right..\label{2.1}\end{align}
\end{lemm}
\proof
If $p=E(x_{0}, s_{0}\textbf{n}(x_{0}))$ is a focal point of $(M,x_{0})$, then there exists a non-zero tangent vector $V\in T_{(x_{0}, s_{0}\textbf{n}(x_{0}))}\mathcal{N}M $ such that $E_{*}V=0.$
Let $\sigma(t)=(x(t), s(t)\textbf{n}(x(t)))$ be a curve in $\mathcal{N}M$ satisfying
$V=\dot{\sigma}(0) $
and $\widetilde{\sigma}(t)=E(\sigma(t))=\Phi(t,s(t))$, then
\begin{align*}
0=E_{*}V=E_{*}(\dot{\sigma}(0))=\dot{\widetilde{\sigma}}(0) & =\Phi_{*}(\frac{\partial}{\partial t}+\dot{s}\frac{\partial}{\partial s})|_{(0,s_{0})}\nonumber\\
                                                & =J(s_{0})+\dot{s}(0)T(s_{0})\nonumber\\
                                                & =E_{1}(s_{0})\mathfrak{s}_{c}^{'}(s_{0})+E_{2}(s_{0})\mathfrak{s}_{c}(s_{0})+\dot{s}(0)T(s_{0}).
\end{align*}
Noting that $g_{T}(T,~E_{i})=0$, we have
\begin{align}
\left\{
\begin{aligned}
E_{1}(s_{0})\mathfrak{s}_{c}^{'}(s_{0})+E_{2}(s_{0})\mathfrak{s}_{c}(s_{0})=0, \\
\dot{s}(0)=0.
\end{aligned}
\right.
\label{2.13}\end{align}
Since $E_i(s)$'s are parallel along the geodesic, their angle and lengths with respect to $g_T$ are constant along the geodesic. Therefore, we have
\begin{align*}\mathfrak{s}_{c}^{'}(s_{0})E_{1}(0)
+\mathfrak{s}_{c}(s_{0})E_{2}(0)=0.\end{align*}
That is,
\begin{align}\mathfrak{s}_{c}^{'}(s_{0})J(0)+\mathfrak{s}_{c}(s_{0})D_{T}^{T}J|_{s=0}=0. \label{2.15}\end{align}
Since  $[\frac{\partial}{\partial s}, \frac{\partial}{\partial t}]=0,$ we obtain
\begin{align} D^{\widetilde T}_{\widetilde T}\widetilde J-D^{\widetilde T}_{\widetilde J}\widetilde T=\Phi_{*}([\frac{\partial}{\partial s}, \frac{\partial}{\partial t}])=0. \label{2.6}\end{align}
Noting that $V\neq0$ and $\dot{s}(0)=0$,  it follows from (\ref{2.15}) that $\mathfrak{s}_{c}(s_{0})\neq0$. Combining (\ref{2.4.1}), (\ref{2.6}) and (\ref{2.15}) yields
\begin{align*}
D^{\textbf{n}}_{X}\textbf{n}=D_{J}^{T}\widetilde T|_{s=0,t=0}=D_{T}^{T}J|_{s=0}=-\frac{\mathfrak{s}_{c}'(s_{0})}{\mathfrak{s}_{c}(s_{0})}X. \end{align*}
Then by Weingarten formula (\ref{0.1}), we have  $A_{\textbf{n}}X=-\left(D^{\textbf{n}}_{X}\textbf{n}\right)^\top_{g_{\textbf{n}}}=\frac{\mathfrak{s}_{c}'(s_{0})}{\mathfrak{s}_{c}(s_{0})}X$. This completes the proof.
\endproof

\subsection{Tubes and Parallel Hypersurfaces}

 Let $\phi:M  \rightarrow \overrightarrow{N}(c)$ be an immersion with codimension $n-m\geq 1$.  If $n-m>1$, we define $M_s~(s>0)$ by the map
\begin{align}\phi_{s}: \mathcal{N}^0M \rightarrow N,~~~~
\phi_{s}(x, \textbf{n})=E(x, s\textbf{n}).\label{2.17}\end{align}
If $(x, s\textbf{n})$ is not a critical point of $E$, then $\phi_{s}$ is an immersion in a neighborhood of $(x, \textbf{n})$ in $\mathcal{N}^0M$. Thus $\phi_{s}$ is an $(n-1)$-dimensional immersion in $N (c)$ if there is no focal point of $M$ on $M_s$. It follows from Lemma \ref{thm1} and \cite{AJ} that for any given $x\in \phi(M)$, there is a neighborhood $U$ of $x$ in $\phi(M)$ such that for all $s>0$ sufficiently small, the restriction of $\phi_{s}$ to $\mathcal{N}^0U$ over $U$ is an immersion onto $M_s$, which lies in a tubular neighbourhood over $U$  and is geometrically a tube of radius $s$ over $U$. For the sake of simplicity, we call $M_s$ a \emph{tube} over $M$ whether it lies in a tubular neighborhood or not.

If $M$ is a hypersurface, i.e. $n-m=1$, then $\mathcal{N}^0M$ is a double covering of $M$. For $s=0$, we have $M_0=\phi(M)$. In this case, for local calculations, we can assume that $M$ is orientable with a local unit normal vector field $\textbf{n}$ and define $M_s$  by the map $\phi_{s}: M \rightarrow N$,
\begin{align}\phi_{s}(x)=E(x, s\textbf{n}),\label{2.18}\end{align}
for $s\in\mathbb{R},$ rather than defining $\phi_{s}$ on the double covering $\mathcal{N}^0M$. Analogously, there is a neighborhood $U$ of $x$ in $\phi(M)$ such that for all $|s|$ sufficiently small, the restriction of $\phi_{s}$ to $U$ is an immersion onto $M_s$, which lies in a tubular neighbourhood over $U$ and is geometrically a parallel hypersurface over $U$. We call $M_s$ a \emph{parallel hypersurface} over $M$ if there is no focal point of $M$ on $M_s$.

\begin{remark} Note that $\overrightarrow{N}(c)$ is probably not backward complete, $s$ is not necessarily to take all negative values. In $N (c)$, which is both forward and backward complete, $s$ can take any real value.  For a well-defined negative value $s$, $M_s=\phi_{s}(M)$ lies locally on the side of $M$ in the direction of $-\textbf{n}$, instead of $\textbf{n}$. But it should be noted that $-\textbf{n}$ may not be  the normal vector of $M$ in a Finsler manifold. So $M_s$ may not be parallel to $M$ in the direction of  $-\textbf{n}$. In fact,  $M$ is parallel to $M_s$ in the direction of $\textbf{n}$, or in other words, $M_s$ is parallel to $M$ in the direction of  $-\textbf{n}$ with respect to the reverse metric
$
\overleftarrow{F}(y)
$. Nevertheless, we also call $M_s$ a parallel hypersurface for convenience.
\end{remark}
We give the principal curvatures of a tube $M_{s}$ in terms of the principal curvatures of the original submanifold $M$ in the following.

\begin{lemm} \label{thm2} Let  $M $ be a submanifold of $\overrightarrow{N}(c)$ and $\phi_{s}$ an immersion near $(x, \textbf{n})\in \mathcal{N}^0M$. Let $\lambda_{1}, \cdots,  \lambda_{_{m}}$ be the principal curvatures of $M$ at $x$ with respect to $\textbf{n}$.  Then the principal curvatures of $M_s$ at $\phi_{s}(x, \textbf{n})$ are
\begin{align} \lambda_{a}(s)&=\frac{-\mathfrak{s}_{c}''(s)+ \lambda_a \mathfrak{s}_{c}'(s)}{\mathfrak{s}_{c}^{'}(s)- \lambda_a \mathfrak{s}_{c}(s)},~~~~a=1,\ldots,m;\\
\lambda_{b}(s)&=\frac{-\mathfrak{s}_{c}^{'}(s)}{\mathfrak{s}_{c}(s)},
~~~~b=m+1,\ldots,n-1~(if~~n-m>1),
\label{2.192}\end{align}
where $\mathfrak{s}_{c}(s)$ is defined by (\ref{2.2}).
\end{lemm}
\proof As $\phi_{s}$ an immersion near $(x, \textbf{n})$, $M_s$ is locally a hypersurface of $\overrightarrow{N}(c)$. We denote the shape operator of $M$ at $x$ and $M_s$ at $E(x, s\textbf{n})$ by $A_{\textbf{n}}$ and $A_{s}$, respectively.

When $n-m\geq1$, as in the preceding subsection, let $\gamma(s)=E(x, s\textbf{n}), s\geq0$, which is a geodesic in $N$. Denote by $X$ the principal vector of $A_{\textbf{n}}$ with respect to $\lambda_{a}$ for any given $a$, $1\leq a\leq m$. We consider the smooth variation of $\gamma$: $$\Phi (t,s)=E(x(t), s\textbf{n}(x(t))),$$ where $x(t)~(-\varepsilon<t<\varepsilon)$ is a smooth curve on $M$ satisfying $x(0)=x,\dot{x}(0)=X$.
Noting that $\phi_{s}(x(t))=\Phi (t,s)$ is a curve on $M_s$, we have
\begin{align}J(s)=\Phi_{*}(\frac{\partial}{\partial t})|_{t=0}=\phi_{s*}X.\label{2.19}\end{align}
Using (\ref{2.19}) and the first equation of (\ref{2.9}), we get
\begin{align}
\phi_{s*}X=\mathfrak{s}_{c}^{'}(s)E_{1}(s)+\mathfrak{s}_{c}(s)E_{2}(s),
\label{2.20}\end{align}
\begin{align}
D^{T}_{T}\phi_{s*}X=\mathfrak{s}_{c}^{''}(s)E_{1}(s)+\mathfrak{s}_{c}'(s)E_{2}(s).
\label{2.201}\end{align}

On the other hand, it follows from the last two equations of (\ref{2.9}) that
\begin{align}
E_{1}(0)&=J(0)=X,\nonumber\\
E_{2}(0)&=  D_{T}^{T}J|_{s=0}= D_{J}^{T}\textbf{n}= -A_{\textbf{n}}X= -\lambda_a X,\nonumber
\end{align}
 which yields $E_{2}(0)=-\lambda_a E_{1}(0)$. Since $E_{i}(s)$'s are parallel vector fields along the geodesic $\gamma(s)$, we know that
\begin{align}
E_{2}(s)=-\lambda_a E_{1}(s).
\label{2.25}\end{align}
Substituting equation (\ref{2.25}) to (\ref{2.20}) and (\ref{2.201}) yields the following
\begin{align}
\phi_{s*}X=(\mathfrak{s}_{c}^{'}(s)- \lambda_a \mathfrak{s}_{c}(s))E_{1}(s),\label{2.26}\\
D^{T}_{T}\phi_{s*}X=(\mathfrak{s}_{c}''(s)- \lambda_a \mathfrak{s}_{c}'(s))E_{1}(s).\label{2.27}
\end{align}
Because $\phi_{s}$ is defined by the normal exponential map, we know that $\widetilde T(s,t)=(d\exp_{x(t)})_{s\textbf{n}}\textbf{n}(x(t))$ is the tangent vector of the radius geodesic from point $x(t)$ and $\phi_{s*}\dot{x}(t)=(d\exp_{x(t)})_{s\textbf{n}}\dot{x}(t)$ is the tangent vector of $M_s$, which is also the tangent vector of the geodesic sphere at $\phi_{s}x(t)$. By Gauss's Lemma (\cite{BCS}),  $$g_{\widetilde T(s,t)}(\widetilde T(s,t),\phi_{s*}\dot{x}(t))=g_{\textbf{n}(x(t))}(\textbf{n}(x(t)),\dot{x}(t))=0,$$
 which shows $\widetilde T(s,t)=\Phi_{*}\textbf{n}(x(t))$ is also a unit normal vector field of $M_s$. From (\ref{2.6}),
 we know that $$ D_{\phi_{s*}X}^{T(s)}\widetilde T|_{t=0}= D_{T(s)}^{T(s)}\phi_{s*}X.$$  Hence by Weingarten formula and (\ref{2.27}), we have  \begin{align}A_s(\phi_{s*}X)=-D_{\phi_{s*}X}^{T}\widetilde T|_{t=0}=(\lambda_a \mathfrak{s}_{c}'(s)-\mathfrak{s}_{c}''(s))E_{1}(s).\label{2.22}\end{align}
Noting that $\phi_{s}(x, \textbf{n})$ is not a focal point of $(M, x)$, we get $\mathfrak{s}_{c}^{'}(s)- \lambda_a \mathfrak{s}_{c}(s)\neq 0$ by (\ref{2.1}). Thus, from (\ref{2.26}) and (\ref{2.22}) we obtain
$$A_{s}(\phi_{s*}X)=\frac{-\mathfrak{s}_{c}''(s)+ \lambda_a \mathfrak{s}_{c}'(s)}{\mathfrak{s}_{c}^{'}(s)- \lambda_a \mathfrak{s}_{c}(s)}\phi_{s*}X,$$
which means that $\phi_{s*}X$ is an eigenvector of $A_{s}$ with principal curvature $\frac{-\mathfrak{s}_{c}''(s)+ \lambda_a \mathfrak{s}_{c}'(s)}{\mathfrak{s}_{c}^{'}(s)- \lambda_a \mathfrak{s}_{c}(s)}$.

If $n-m>1$, let $\sigma(t)~(-\varepsilon<t<\varepsilon)$ be any curve in $\mathcal{N}_x^0M$ such that $\sigma(0)=\textbf{n},\dot\sigma(0)=V$. Similarly, we consider the smooth variation of $\gamma$: $$\Phi (t,s)=E(x, s\sigma(t)).$$ Since $\phi_{s}(\sigma(t))=\Phi (t,s)$ is a curve on $M_s$, the variation vector field $J(s)$ satisfies
\begin{align}J(0)=0,~~J(s)=\Phi_{*}(\frac{\partial}{\partial t})|_{t=0}=\phi_{s*}V,~~g_{T(s)}(J(s), ~~T(s))=0.\label{2.191}\end{align}
 Noting that $J(s)$ is a Jacobi field along the geodesic $\gamma$, we can get from (\ref{2.191}) that
\begin{align}
\left\{
\begin{aligned}
J(s) & = \mathfrak{s}_{c}(s)E(s),  \\
E(0) & = D_{T}^{T}J|_{s=0},
\end{aligned}
\right.
\label{2.91}\end{align}
where $E(s)$ is a parallel vector field along $\gamma(s)$ satisfying $g_T(E, T)=0.$ From (\ref{2.191})  and (\ref{2.91}), we obtain
\begin{align*}
D_{J}^{T}\widetilde T|_{t=0}&=D_{T}^{T}J=\mathfrak{s}_{c}'(s)E(s), \\
\phi_{s*}V&= \mathfrak{s}_{c}(s)E(s).\end{align*}
 So then
\begin{align*}
A_{s}(\phi_{s*}V)=-D_{J}^{T}\widetilde T|_{t=0}=-\frac{\mathfrak{s}_{c}'(s)}{\mathfrak{s}_{c}(s)}\phi_{s*}V.
\end{align*}
\endproof

\section{Isoparametric Hypersurfaces in Finsler space form}

\subsection{ The reverse metric}
Let $(N, F)$ be an oriented smooth Finsler manifold. The reverse metric of $F$ is defined by
$
\overleftarrow{F}(x,y)=F(x,-y)$, where $y\in T_{x}N$, $x\in N$. Then $
\overleftarrow{F}^*(\xi)=F^*(-\xi)
$, where $\xi\in T^*_{x}N$.
One can easily verify that
\begin{align*}
\overleftarrow{g}_{ij}(y)=g_{ij}(-y),~~\overleftarrow{\Gamma}^{i}_{jk}(y)
=\Gamma^{i}_{jk}(-y),~~\overleftarrow{N}^{i}_{j}(y)=-N^{i}_{j}(-y).
\end{align*}
Let $f: N \to \mathbb{R}$ be a smooth function. Then from the definition of gradient vector and Finsler-Laplacian of $f$, we have
\begin{align}
\overleftarrow{\nabla}(-f)=-\nabla f,~~\overleftarrow{\Delta} (-f)=-\Delta f\label{3.05}.
\end{align}
\begin{lemm} \label{lem4.01}Let $f: N \to \mathbb{R}$ be a non-constant function defined on a Finsler manifold $(N, F)$. Then $f$ is an isoparametric function with respect to $F$ if and only if $-f$ is an isoparametric function with respect to $\overleftarrow{F}$. Moreover, the principal curvatures of isoparametric hypersurfaces in terms of the level sets of $f$ and $-f$ with respect to $F$ and $\overleftarrow{F}$ respectively are quite the contrary.
\end{lemm}
\begin{proof}
According to the definition of isoparametric function in Finsler space, if $f$ is an isoparametric function with respect to $F$, (\ref{0.01}) holds. Combining (\ref{3.05})  with (\ref{0.01}), we have
\begin{equation*}\left\{\begin{aligned}
      &\overleftarrow{F}(\overleftarrow{\nabla} (-f))= a (f),\\
      &\overleftarrow{\Delta} (-f)=- b (f).
\end{aligned}\right.\end{equation*}
Thus $-f$ is an isoparametric function with respect to $\overleftarrow{F}$.

Conversely, $F$ can be regarded as the reverse metric of $\overleftarrow{F}$, hence the proof is analogous.

Furthermore, let $M$ be a level surface of $f$. Then $\textbf{n}=\frac{\nabla f}{F(\nabla f)}$ and $-\textbf{n}=\frac{\overleftarrow{\nabla} (-f)}{\overleftarrow{F}(\overleftarrow{\nabla} (-f))}$ are the unit normal vector fields of~$M$ with respect to $F$ and $\overleftarrow{F}$, respectively. For any $X\in T_{x}M$ and $x\in M$, we have
\begin{align*}
\overleftarrow{D}^{\textbf{-n}}_{X}(-\textbf{n})= -D^{\textbf{n}}_{X}\textbf{n},
\end{align*}
where $D$ and $\overleftarrow{D}$ denote the covariant derivatives with respect to $F$ and $\overleftarrow{F}$, respectively. Noting that $D^{\textbf{n}}_{X}\textbf{n}=-A_{\textbf{n}}X=-\lambda X$, we complete the proof.
\end{proof}
\subsection{The integrability of principal foliations}

Let $M$ be an isoparametric hypersurface in a Finsler space form $N(c)$.
Suppose that $\lambda$ and $\mu$ are the constant principal curvatures with corresponding principal foliations $V_{\lambda}$ and $V_{\mu}$. If $X \in V_{\lambda}$ and $Y \in V_{\mu}$, then from (\ref{1.15}), one easily verifies that\\
\begin{align}\label{3.5}
\hat g(({ \hat{\nabla}}_Z{A}_{\textbf{n}})X,Y)=(\lambda-\mu)\hat g({ \hat{\nabla}}_ZX, Y)
\end{align}
for any vector $Z\in TM$.
\begin{lemm} Let $(M,\hat g)$ be an isoparametric hypersurface in a Finsler space form $N^{n}$ of constant flag curvature $c$. For all principal curvaturs $\lambda, \mu$, we have\\
$(1)$ The principal foliations is integrable, that is $[X, Y] \in V_{\lambda}$ for all $X, Y \in V_{\lambda}$.\\
$(2)$ ${ \hat{\nabla}}_XY \perp V_{\lambda}$ if $X \in V_{\lambda},~~Y \in V_{\mu},~~\lambda \neq \mu.$
\end{lemm}
\proof Let $X$ and $Y$ be in $V_{\lambda}$ and $Z \in V_{\mu}$ for $\mu \neq \lambda$. \\
(1). By the Codazzi equation (\ref{1.15}) and (\ref{3.5}),
\begin{eqnarray}
0&=&\hat g(({ \hat{\nabla}}_X{A}_{\textbf{n}})Z-({ \hat{\nabla}}_Z{A}_{\textbf{n}})X, Y)\nonumber\\
\label{3.6}&=&(\mu-\lambda)\hat g({ \hat{\nabla}}_XZ, Y)\\
&=&(\mu-\lambda)(X\hat g(Z, Y)-\hat g({ \hat{\nabla}}_XY, Z)+2C_{\textbf{n}}({A}_{\textbf{n}}X,Y,Z)+2L_{\textbf{n}}(X,Y,Z))\nonumber\\
&=&(\lambda-\mu)\hat g({ \hat{\nabla}}_XY, Z)-(\lambda-\mu)(2C_{\textbf{n}}({A}_{\textbf{n}}X,Y,Z)+2L_{\textbf{n}}(X,Y,Z)).\label{3.7}
\end{eqnarray}
Thus
\begin{eqnarray}\label{3.8}
\hat g({ \hat{\nabla}}_XY, Z)=2C_{\textbf{n}}({A}_{\textbf{n}}X,Y,Z)+2L_{\textbf{n}}(X,Y,Z).
\end{eqnarray}
Similarly,\\
\begin{eqnarray}\label{3.9}
\hat g({ \hat{\nabla}}_YX, Z)=2C_{\textbf{n}}(X,{A}_{\textbf{n}}Y,Z)+2L_{\textbf{n}}(X,Y,Z).
\end{eqnarray}
So we have $\hat g({ \hat{\nabla}}_YX-{ \hat{\nabla}}_XY, Z) = 0$, that is $\hat g([X, Y], Z) = 0$. Therefore $[X, Y] \in V_{\lambda}$.\\
(2). From (\ref{3.6}), we know $\hat g({ \hat{\nabla}}_XZ, Y ) = 0$ and thus ${ \hat{\nabla}}_XZ \perp  V_{\lambda}$.
\endproof

\subsection{ Focal submanifold}
For an isoparametric hypersurface in $N(c)$, the exponential map $E$ is smooth and the multiplicity of every principal curvature is constant. It follows from Lemma \ref{thm1} that the rank of $dE$ is constant. And because the principal foliations is also integrable, as in the Riemannian case, we can give a manifold structure to a sheet of the focal set, which is called \emph{focal submanifold}.

Let $M$ be a connected, oriented isoparametric hypersurface in $N (c)$ with unit normal vector field $\textbf{n}$ and $g$ distinct constant principal curvatures, say
\begin{align*}
\lambda_{1},~\lambda_{2},~\cdots, ~\lambda_{g}.
\end{align*}
We denote the multiplicity and the corresponding principal foliation of $\lambda_{i}$ by $m_{i}$ and $V_{i}$, respectively,  then $m_{1}+m_{2}+\cdots+m_{g}=m=n-1$. From Lemma \ref{thm1} and (\ref{2.26}), if $S$ is a focal submanifold  of~$M$, then there exists $s_{i}=s(\lambda_{i}, c)$ such that \begin{align}\lambda_{i}=\frac{\mathfrak{s}_{c}'(s_i)}{\mathfrak{s}_{c}(s_i)},
~~~~\phi_{s_i}M=S,~~~~\phi_{s_i*}V_{i}=0,\label{3.17}\end{align}where $\mathfrak{s}_{c}(s)$ is defined by (\ref{2.2}). Denote by $S_{i}=\phi_{s_{i}}M$ the focal submanifold of~$M$ corresponding to $\lambda_{i}$. The principal curvatures of
a focal submanifold can be expressed in analogous forms as in the Riemannian case.
\begin{lemm} \label{lem4.2} Let $M$ be a connected isoparametric hypersurface in $N (c)$ and $S_{i}$ be a focal submanifold of~$M$. Then for every unit normal vector $\eta$ at any $p\in S_{i}$, the shape operator $A_{\eta}$ has principal curvatures $\frac{c+\lambda_{i}\lambda_{j}}{\lambda_{i}-\lambda_{j}}$ with multiplicities $m_{j}$ and corresponding principal vectors $\phi_{s_{i}*}X,$ where~$X\in V_{j}(x)$ and $j\neq i, 1\leq i,j\leq g.$
\end{lemm}
\proof Let $\textbf{n}$ be the unit normal vector field of $M$. From Lemma \ref{lem4.01} and Lemma \ref{thm2}, we know that in $(N,\overleftarrow{F})$, $M$ is the tube over $S_{i}$ with respect to $-\textbf{n}$. Then for any unit normal vector $\eta$ at  $p\in S_{i}$ with respect to $F$,  $-\eta$  is a unit normal vector at $p\in S_{i}$ with respect to $\overleftarrow{F}$. Moreover, there exists a point $x\in M$ such that $x=\overleftarrow{\phi}_{s_{i}}(p,-\eta)$, where $\overleftarrow{\phi}_{s_{i}}: \overleftarrow{\mathcal{N}^0} S_{i}\rightarrow N$ is defined by
(\ref{2.17}) with respect to $\overleftarrow{F}$. Thus $p=\phi_{s_{i}}(x)$.
By the proof of Lemma \ref{thm2}, we know that even at the focal point of $(M, x)$, if $\mathfrak{s}_{c}^{'}(s_i)- \lambda_j \mathfrak{s}_{c}(s_i)\neq 0$, we still have
\begin{align}
D_{\phi_{s*}X}^{\eta}\eta=&\frac{\mathfrak{s}_{c}''(s_{i})- \lambda_{j} \mathfrak{s}_{c}'(s_{i})}{\mathfrak{s}_{c}^{'}(s_{i})-\lambda_{j} \mathfrak{s}_{c}(s_{i})}\phi_{s_i*}X,~~\forall X\in V_{j}(x).\label{3.18}
\end{align}
Hence by Weingarten formula (\ref{0.1}) and (\ref{3.18}), \begin{align}A_{\eta}(\phi_{s*}X)&=-\left[D_{\phi_{s*}X}^{\eta}\eta\right]^{\top}_{g_\eta}\nonumber\\
&=-D_{\phi_{s*}X}^{\eta}\eta\label{2.202}\\
&=\frac{-\mathfrak{s}_{c}''(s_{i})+ \lambda_{j} \mathfrak{s}_{c}'(s_{i})}{\mathfrak{s}_{c}^{'}(s_{i})- \lambda_{j} \mathfrak{s}_{c}(s_{i})}\phi_{s_i*}X,~~\forall X\in V_{j}(x),~~j\neq i,\nonumber\end{align}
where $1\leq i,j\leq g.$
The desired conclusion follows by (\ref{3.17}), (\ref{3.18}) and  (\ref{2.2}).
\endproof

\begin{theo} \label{thm3} Let $M\hookrightarrow N (c)$ be a connected isoparametric hypersurface. If the first class reversible torsion of $N (c)$ vanishes, then each focal submanifold of~$M$ is an anisotropic-minimal submanifold in $N (c)$.
\end{theo}
\proof
Let $\eta$ be a unit normal vector to a focal submanifold $S_{0}$ of $M$. Then $\eta_{-}$ is also a unit normal vector to $S_{0}$ and there exist two points $x_1,x_2\in M$ such that $p=\phi_{s_{i}}(x_1)=\phi_{s_{i}}(x_2)$ and $\eta=\phi_{s_{i}*}\textbf{n}(x_1), \eta_{-}=\phi_{s_{i}*}\textbf{n}(x_2)$. By Lemma \ref{lem4.2}, the shape operators $A_{\eta}$ and $A_{\eta_{-}}$ have the same eigenvalues with the same multiplicities. So
$$\text{tr} A_{\eta}=\text{tr} A_{\eta_{-}}.$$
 If $\tau_{\eta}(X)=D^{\eta}_{X}\eta+D^{\eta_{-}}_{X}\eta_{-}=0,$
 from (\ref{3.18}) and (\ref{2.202}), for any eigenvector $X$ of $A_{\eta_{-}}$, we have $$A_{\eta_{-}}X=-\left[D_{X}^{\eta_-}\eta_-\right]^{\top}_{g_{\eta_-}}=-D^{\eta_{-}}_{X}\eta_{-}
=D^{\eta}_{X}\eta=\left[D_{\phi_{s*}X}^{\eta}\eta\right]^{\top}_{g_\eta}
=-A_{\eta}X.$$ That is, the principal curvatures of $A_{\eta_{-}}$ and $A_{\eta}$ are all actually opposite. Hence $\text{tr}  A_{\eta}=0$. Since this is true for all unit normal vectors $\eta$, we conclude that $S_{0}$ is an anisotropic-minimal submanifold in $N (c)$.
\endproof

\subsection{ Proof of Theorem \ref{thm00}}
\proof
(1)~ Let $\nu=\mathcal L\eta$,  $\{e_{a}\},a=1, 2, \ldots, m=n-1-m_i$ (resp. $\{\bar{e}_{a}\}$) be the orthonormal principal vectors of $S_i$ with respect to $g_{\eta}$ (resp. $g_{\eta_{-}}$) and ${\mu}_{a}$ (resp. $\bar{\mu}_{a}$) be the corresponding principal curvatures. We have
\begin{equation*}\aligned
0
&=e_{b}(g_{\eta}(\eta, e_{a}))=g_{\eta}(\nabla^{\eta}_{e_{b}}\eta, e_{a})+g_{\eta}(\eta, \nabla^{\eta}_{e_{b}}e_{a})=g_{\eta}(\nabla^{\eta_-}_{e_{b}}\eta, e_{a})+g_{\eta}(\eta, \nabla^{\eta_-}_{e_{b}}e_{a}).
\endaligned\end{equation*}
If $\tau'_{\eta}=0,
$ then $\nabla^{\eta}_{e_{b}}\eta=\nabla^{\eta_-}_{e_{b}}\eta.$ Thus
\begin{equation*}
g_{\eta}(\eta, \nabla^{\eta}_{e_{b}}e_{a})=g_{\eta}(\eta, \nabla^{\eta_-}_{e_{b}}e_{a})
=\nu(\nabla^{\eta_-}_{e_{b}}e_{a})=-F^{*}(-\nu)g_{\eta_-}(\eta_-, \nabla^{\eta_-}_{e_{b}}e_{a}).
\end{equation*}
 That is
\begin{equation}\label{c2}
g_{\eta_{-}}(\nabla^{\eta_{-}}_{ e_{b}}\eta_{-}, e_{a})
=-\frac{1}{F^{*}(-\nu)}g_{\eta}(\nabla^{\eta}_{ e_{b}}\eta,  e_{a}).
\end{equation}
Set $e_{a}=u^{b}_{a}\bar{e}_{b}$. By the Weingarten formula on $S_i$, we obtain
\begin{equation}\aligned\label{c12}
g_{\eta}(\nabla^{\eta}_{e_{b}}\eta, e_{a})
&=g_{\eta}(-A_{\eta}e_{b}, e_{a})=-{\mu}_{b}\delta_{ab},\\
g_{\eta_{-}}(\nabla^{\eta_{-}}_{e_{b}}\eta_{-}, e_{a})&=u^{c}_{b}u^{d}_{a} g_{\eta_{-}}(\nabla^{n_{-}}_{\bar{e}_{c}}\eta_{-}, \bar{e}_{d})=-\bar{\mu}_{c}u^{c}_{b}u^{d}_{a}\delta_{cd}.
\endaligned\end{equation}
Plugging \eqref{c12} into \eqref{c2} yields
\begin{equation*}
-{\mu}_{a}=F^{*}(-\nu)\sum\limits_{b}(u^{b}_{a})^{2}\bar{\mu}_{b}.
\end{equation*}
From Lemma \ref{lem4.2}, we know that ${\mu}_{a}$ and $\bar{\mu}_{a}$ may be different, but $\{{\mu}_{a}\}$ and $\{\bar{\mu}_{a}\}$ are the same set. Then
\begin{equation*}
-\sum\limits_{b}{\bar\mu}_{b}=-\sum\limits_{a}{\mu}_{a}=F^{*}(-\nu)\sum\limits_{a,b}(u^{b}_{a})^{2}\bar{\mu}_{b}
=F^{*}(-\nu)\sum\limits_{b}\bar{\mu}_{b}\left[\sum\limits_{a}(u^{b}_{a})^{2}\right].
\end{equation*}
So
\begin{equation}\label{c5}
\sum\limits_{b}{\bar\mu}_{b}\left[F^{*}(-\nu)\sum\limits_{a}(u^{b}_{a})^{2}+1\right]=0.
\end{equation}
It is obvious that $F^{*}(-\nu)\sum\limits_{a}(u^{b}_{a})^{2}+1\geq1$.
\eqref{c5} holds at a given point. From Lemma \ref{lem4.2} and \eqref{c5}, we have \eqref{c1}.

(2)~If $\tau_{\eta}=0,$ then $\text{tr}A_{\eta}$ is zero for each $i$ and  each unit normal vector $\eta$ on the focal submanifold $S_{i}$ by Theorem \ref{thm3}. It follows that \begin{equation}\label{c20}0=\text{tr}  A_{\eta}=\sum_{j\neq i}m_{j}\frac{c+\lambda_{i}\lambda_{j}}{\lambda_{i}-\lambda_{j}}.\end{equation}
The last conclusion of Theorem \ref{thm00} can be obtained immediately from Theorem \ref{thm3} and \eqref{c20}.
\endproof
\subsection{Proof of Theorem \ref{thm0}}
\proof
As in Riemannian geometry, we also divide the proof into two situations.\\
\textbf{Case I}: $\textbf{c}\leq \textbf{0}$.

From general Cartan formula \eqref{c1}, we can get the results as in \cite{DH}.
\\
\textbf{Case II}: $\textbf{c} > \textbf{0}$ and $\tau=0$.

By $(\ref{3.17}$) and $(\ref{2.2}$), we suppose that $M$ has $g$ distinct principal curvatures $\lambda_{i}=\frac{\mathfrak{s}_{c}'(s_i)}{\mathfrak{s}_{c}(s_i)}=\sqrt{c} \cot\theta_{i}$, $0<\theta_{1}<\cdots< \theta_{g}< \pi$, with respective multiplicities $m_{i}$ and $\theta_{i}=\sqrt{c}s_{i}$ (or $\theta_{i}=\sqrt{c}s_{i}+\pi$). From Cartan-type formula \eqref{c20}, we can get the results as in the Riemannian case ([7,Theorem 3.26]).
\endproof
\begin{remark} Unlike those on the Euclidean sphere $\mathbb S^n$, the  geodesics in a Finsler space form with positive constant flag curvature are not necessarily closed. So the number of the focal points of $(M,x)$ along every normal geodesic to $M$  through $x$ may be more than $2g$.  In general, if $M$ is an isoparametric hypersurface of $(N,F)$ with respect to the normal vector $\textbf{n}$, it is not necessarily isoparametric with respect to the normal vector $\textbf{n}_-$. But that is true when~$\tau=0$. There are probably two different normal geodesics through $x\in M$ and each of them has the characteristics described in Theorem \ref{thm0}.
\end{remark}
\section{ Reversible torsion of Randers spaces}
Let $(N,F,d\mu_{\text{BH}})$ be an $n-$dimensional Randers space with BH-volume form and let its navigation expression be
$$F=\frac{\sqrt{\lambda h^2+W_{0}^2}-W_{0}}{\lambda}=\frac{\sqrt{\lambda h_{ij}y^iy^j+(W_{i}y^i)^2}-W_{i}y^i}{\lambda},$$
where $\lambda=1-\|W\|_h^2, W=W^i\frac{\partial}{\partial x^i}, W_{i}=h_{ij}W^{j}$.
Let $h^*$ be the dual metric of $h$. By [13, p.39-40], the dual metric $F^*$ can be expressed as
\begin{align}
F^*=&h^*+\xi(W)=\sqrt{h^{ij}\xi_i\xi_j}+W^{i}\xi_i,\quad \xi=\xi_idx^i \in T^*_x N.\label{4.1} \end{align}
Denote
$$s_{ij}=\frac{1}{2}(W_{i|j}- W_{j|i}),~~~s_{j}=W^{i}s_{ij},~~~s^{i}=h^{ij}s_{j},$$
$$s^{i}_{j}=h^{ik}s_{kj},~~~s^{i}_{0}=s^{i}_{j}y^{j},$$
where ¡°$|$¡± denotes the covariant differential with respect to $h$. From [16, Theorem 5.10], $F$ has vanishing $\mathbf{S}$-curvature if and only if $W$ is a Killing vector field. In this case, we have
\begin{align}
W_{i|j}&=-W_{j|i},~~s_{ij}=W_{i|j},~~s^{i}_{j}=W^{i}_{|j},\nonumber\\
G^{i}&=\bar{G}^{i}-Fs^{i}_{0}-\frac{1}{2}F^{2}s^{i},\label{3.25}
\end{align}
where $G^{i}$ and $\bar{G}^{i}$ are the geodesic coefficients of $F$ and $h$, respectively. Then
\begin{align}
N^{i}_{j}&=\frac{\partial G^{i}}{\partial y^{j}}=\bar{N}^{i}_{j}-F_{y^{j}}s^{i}_{0}-Fs^{i}_{j}-FF_{y^{j}}s^{i},\label{3.26}
\end{align}
where $\bar{N}^{i}_{j}=\frac{\partial \bar G^{i}}{\partial y^{j}}=\bar \Gamma^i_{jk}(x)y^k$ and $\bar \Gamma^i_{jk}$'s
 are the Levi-Civita connection coefficients of $h$.
\begin{lemm} \label{lem6.0} Let $(N,F,d\mu_{\text{BH}})$ be a Randers space with navigation data $(h,W)$. If its $\mathbf{S}$-curvature  vanishes, or equivalently $W$ is a Killing vector, then its reversible torsion vanishes.
\end{lemm}
\proof For $x\in N$, let $\eta$ be a unit vector field in the neighbourhood of $x$ and let $X\in T_{x}N$  satisfy $g_{\eta}(\eta,X)=0$. If $\xi=\mathcal L \eta$, then $\xi(X)=0$. That is, $F_{y^{j}}(\eta)X^j=0$ and $F_{y^{j}}(\eta_{-})X^j=0$. Set $\eta_{+}=\eta$, from (\ref{Z01}) and (\ref{4.1}), we know that
$$\eta_{-}=\frac{{\mathcal L}^{-1}(-{\mathcal L}\eta)}{F^{*}(-{\mathcal L}\eta)}=\frac{{\mathcal L}^{-1}(-\xi)}{F^{*}(-\xi)},~~~~\eta_{\pm}^i=F^{*}_{\xi_i}(\pm\xi)=\frac{\pm h^{ij}\xi_j}{h^{*}(\xi)}+W^i.
$$
Then by (\ref{3.19}), (\ref{3.25}) and (\ref{3.26}), we have
\begin{align*}
\tau_{\eta}(X)&=D^{\eta}_{X}\eta+D^{\eta_{-}}_{X}\eta_{-}\\
&=(\eta^{i}_{x^{j}}+N^{i}_{j}(\eta)+[\eta_{-}^{i}]_{x^{j}}+N^{i}_{j}(\eta_{-}))X^{j}\frac{\partial}{\partial x^{i}}\nonumber\\
&=(2W^{i}_{x^{j}}+\bar N^{i}_{j}(\eta)+\bar N^{i}_{j}(\eta_{-})-2s^i_j)X^{j}\frac{\partial}{\partial x^{i}}\nonumber\\
&=(2W^{i}_{x^{j}}+\bar \Gamma^i_{jk}(\eta^k+\eta_{-}^k)-2s^i_j)X^{j}\frac{\partial}{\partial x^{i}}\nonumber\\
&=(2W^{i}_{x^{j}}+2\bar \Gamma^i_{jk}W^k-2s^i_j)X^{j}\frac{\partial}{\partial x^{i}}\nonumber\\
&=2(W^{i}_{|j}-s^i_j)X^{j}\frac{\partial}{\partial x^{i}}\nonumber\\
&=0.
\end{align*}
\endproof

\begin{remark} From \cite{BRS} and \cite{SS}, the Randers space $(N,F,d\mu_{\text{BH}})$ has constant flag curvature $c$ if and only if the Riemannian space $(N, h)$ has constant sectional curvature $\bar c$ and $W$ is a homothetic vector field. In this case, $F$ has constant $\mathbf{S}$-curvature $c'$ and similarly we can prove that $\tau_{\eta}(X)=2c'X.$ So $(N,F,d\mu_{\text{BH}})$ has vanishing reversible torsion iff $c'=0$. If $\bar c\neq0$, then $c'=0$ and $c=\bar c$. That is, $(N,F,d\mu_{\text{BH}})$ must have vanishing $\mathbf{S}$-curvature. If $\bar c=0$, there are also many nontrivial Randers space forms with vanishing $\mathbf{S}$-curvature besides Minkowski spaces.
\end{remark}

Qun He \\
School of Mathematical Sciences, Tongji University, Shanghai, 200092, China\\
E-mail: hequn@tongji.edu.cn\\

Yali Chen\\
School of Mathematical Sciences, Tongji University, Shanghai, 200092, China\\
E-mail: chenyl90@tongji.edu\\

Songting Yin\\
Department of Mathematics and Computer Science, Tongling University, Tongling 244000,
China\\
E-mail: yst419@163.com\\

Tingting Ren\\
School of Mathematical Sciences, Tongji University, Shanghai, 200092, China\\
E-mail: 1531938@tongji.edu.cn

\end{document}